\documentclass[11pt,twoside,a4paper,reqno,leqno]{amsart}
\usepackage[latin1]{inputenc}
\usepackage{amssymb}
\usepackage{cite}
\usepackage{graphicx}

\usepackage{epsfig}
\input{epsf.sty}

\newtheorem{theorem}           {Theorem}
\newtheorem{prop+}           {Proposition}
\newtheorem{coro+}           {Corollary}
\newtheorem{lemma}           {Lemma}
\newtheorem{conjecture}      {Conjecture}

\theoremstyle{definition}
\newtheorem{rema+}           {Remark}
\newtheorem{defi+}           {Definition}

\newtheorem{corollary}[theorem]{Corollary}

\def\T{{T}}
\def\S{{S}}
\def\A{{A}}
\def\Z{{Z}}
\newcommand\al{\alpha}
\newcommand\vpil[1]{\overset{\leftarrow}{#1}}
\newcommand\hpil[1]{\overset{\rightarrow}{#1}}
\def\notto{\mathrel{\not\to}}
\newcommand\fpn{\ensuremath{f_n(p)}}
\newcommand\hpn{\ensuremath{h_n(p)}}
\newcommand\gnp{\ensuremath{G(n,p)}}
\newcommand\gnm{\ensuremath{G(n,m)}}
\newcommand\ggnm{{\ensuremath{\vec G(n,m)}}}
\newcommand\ggnp{{\ensuremath{\vec G(n,p)}}}
\renewcommand\P{\operatorname{\mathbb P{}}}

\newcommand\nn{\binom n2}

\newenvironment{romenumerate}[1][0pt]{
\addtolength{\leftmargini}{#1}\begin{enumerate}
 }{\end{enumerate}}

\title{Correlation of paths between distinct vertices in a randomly oriented graph.}

\author{Madeleine Leander}
\date{\today}
\address{Department of Mathematics, Stockholm University, SE-106 91
  Stockholm, Sweden}
\email{madde@math.su.se}

\author{Svante Linusson}
\address{Department of Mathematics, KTH-Royal Institute of Technology, SE-100 44
  Stockholm, Sweden}
\email{linusson@math.kth.se}

\begin{document}

\begin{abstract}
We prove that in a random tournament the events $\{s\rightarrow a\}$ and $\{t\rightarrow b\}$ are positively correlated, for distinct 
vertices $a,s,b,t \in K_n.$ It is also proven that the correlation between the events $\{s\rightarrow a\}$ and $\{t\rightarrow b\}$ in the 
random graphs $G(n,p)$ and $G(n,m)$ with random orientation is positive for every fixed $p>0$ and sufficiently large $n$ 
(with $m=\left\lfloor p \binom{n}{2}\right\rfloor$). We conjecture it to be positive for all $p$ and all $n$. 
An exact recursion for $\P(\{s\rightarrow a\} \cap \{t\rightarrow b\})$ in $\gnp$ is given. 
\end{abstract}

\maketitle

\section{Introduction}

Let $G$ be a graph on $n$ vertices and $a,b,s,t \in V(G)$ four different vertices in the graph. Let further every edge in $G$ be oriented either way with the same probability independently of each other. This model was first considered in 
\cite{M81}, and a similar model was discussed in \cite{G00}. We will study the correlation between the two events $\{s\rightarrow a\}$ and $\{t\rightarrow b\}$. Our main result is that these events are positively correlated for the complete graph and for two natural models of random graphs. Note however that it is easy to construct examples when the correlation will be negative, e.g. if $G$ is the path on four vertices with edges $sb,ba,at$. 

The events $\{s\rightarrow a\}$ and $\{s\rightarrow b\}$ can be shown to have positive correlation for any 
vertices in any graph  $G$. In \cite{AL} it was  proven, somewhat surprisingly, that also the events $\{s\rightarrow a\}$
and $\{b\rightarrow s\}$ have positive correlation in $K_n$, when $n\ge 5$, but negative correlation if 
$G$ is a tree or a cycle. Further, in \cite{AJL} it was shown that in the random graph models $\gnp$ and
$\gnm$ for a fixed probability $p$ ($= m/\nn)$  and large enough $n$ the correlation between 
$\{s\rightarrow a\}$ and $\{b\rightarrow s\}$ is negative if $p$ is below a critical value and positive if $p$ is above the critical value. The critical value in $\gnp$ was exactly $1/2$ and  in $\gnm$ approx. $0.799$. 

The situation in this paper turns out to be different. We prove positive correlation when $G$ is $K_n$ and 
in $\gnp$ and $\gnm$ for fixed $p>0$ and $n$ sufficiently large. We conjecture that it is in fact non-negative for all pairs $n,p$.

For technical reasons we will study the complementary events $A:=\{s\nrightarrow a\},$ the event that no directed path from $s$ to $a$ exists, and $B:=\{t\nrightarrow b\}$.  Note that the events $A$ and $B$ have the same covariance as the events $\{s\rightarrow a\}$ and $\{t\rightarrow b\}$.

The paper is organized as follows. In section \ref{corrKn} we present a lower bound for $\P(A\cap B)$ and prove that $A$ and $B$ are positively correlated for $n\geq 4.$ An intuitive explanation is due to the fact that the biggest terms of $\P(A)$ comes from when no edges are directed from $s$ and when no edges are directed to $a,$ analogously for $\P(B).$ We also show that the relative covariance of the two events converges to $2/3$ as $n\rightarrow \infty.$ 

In section \ref{gnp} we consider the random graph $G(n,p)$ on $n$ vertices. It is a random graph model in which every edge exists with probability $p$ independently of each other and then every existing edge is directed in either of the two directions with the same probability independently of all other edges. 
Note that the two random prosesses can be combined in two different ways. In this paper we study the joined probability space of $G(n,p)$ and that of egde orientations, which we call $\ggnp$. This will be refered to as
 the annealed version. The other possibility, the quenched version, will be briefly discussed in section \ref{quenched}. We prove that for fixed $p>0$ and sufficiently large $n$ the events $A$ and $B$ will be positively correlated in $\ggnp.$

In Section \ref{gnm} we study the random graph model $\gnm$, with uniform distribution among
all graphs with $n$ vertices and $m$ edges. Note that in this graph the edges does not exist independently 
of each other since the number of edges in the graph is fixed. As before every existing edge is directed 
in either way with equal probability independent of all other edges. We prove that for fixed $p=m/\nn$the events $A$ and $B$ are positively correlated for sufficiently large $n$. 

In Section \ref{S:rec} we give an exact recursion to compute $\P(A\cap B)$ in $\gnp$ which supports our conjecture that the correlation is positive for all values of $n$ and $p$.

The problems studied here was first motivated by the, so far in vain, attempts to prove the so called
bunkbed conjecture, see \cite{L1}.

\section{Correlation in a random tournament}
\label{corrKn}

To show that the correlation between $A$ and $B$ is positive we need a sufficient upperbound for $\P(A)$ (and $\P(B)$)  and a lower bound for $\P(A\cap B).$ Both an upper bound and a lower bound for $\P(A)$ was given in \cite{AL}:
\begin{lemma}[Theorem 2.1 in \cite{AL}]
\label{sl}
For all $n\geq 2$,
$$\left( \frac12 \right) ^{n-2} \left(1- \left(\frac12 \right) ^{n-1} \right)\leq \P(A) \leq \left(\frac12\right) ^{n-2}\left(1+3.2\cdot \left(\frac78\right) ^{n-1}\right).$$
\end{lemma}

The next lemma gives a lowerbound for the probability of the event $A\cap B.$ 

\begin{lemma}
\label{lowerB}
For all $n\geq 4$,
$$\P(A\cap B)\geq \left(\frac12\right)^{2n-4} \left( 3 - \left(\frac12\right)^{2n-7}-\left(\frac12\right)^{n-4}\right)$$ 

\end{lemma}

\begin{proof}
Define $I_{a,b}$ to be the set of points in $[n]\backslash \{a,b\}$ that can reach $a$ or $b$ in one step, that is with a single edge directed to $a$ or $b$. Similarily define $O_{s,t}$ to be the set of points in $[n]\backslash \{s,t\}$ that can be reached from $s$ or $t$ in one step. Define further $I_a$ and $I_b$ to be the set of points in $[n]\backslash \{a\}$ and $[n]\backslash \{b\}$ respectively that can reach $a$ and $b$ respectively in one step, and finally in the same way define $O_s$ and $O_t$ to be the set of points in $[n]\backslash \{s\}$ and $[n]\backslash \{t\}$ respectively that can be reached from $s$ and $t$ respectively in one step. 

The four events $I_{a,b}=\emptyset, O_{s,t}=\emptyset, I_a=O_t=\emptyset$ and $I_b=O_s=\emptyset$ all implies $A\cap B$. Hence we have $$\P(A\cap B)\geq \P((I_{a,b}=\emptyset)\cup(O_{s,t}=\emptyset)\cup(I_a=O_t=\emptyset)\cup(I_b=O_s=\emptyset)).$$ 
By inclusion-exclusion we have $$\P((I_{a,b}=\emptyset)\cup(O_{s,t}=\emptyset)\cup(I_a=O_t=\emptyset)\cup(I_b=O_s=\emptyset))=$$
\begin{eqnarray*}
&=& 2\cdot \left(\frac12 \right) ^{2(n-2)} +2\cdot\left(\frac12 \right) ^{2(n-1)-1}  -\\
&-& \left( \left(\frac12 \right) ^{4(n-2)-4} +4\cdot \left(\frac12 \right) ^{3n-6} \right) + \\
&+&2\cdot \left(\frac12 \right) ^{4n-10}=\\
&=& \left(\frac12 \right) ^{2n-4}\left( 3- \left(\frac12 \right) ^{2n-7}  - \left(\frac12 \right) ^{n-4} \right) 
\end{eqnarray*}
since the events $(I_a=\emptyset)$ and $ (I_b=\emptyset)$ are disjoint and so are the events 
$(O_s=\emptyset)$ and $(O_t=\emptyset)$. 
\end{proof}

\begin{theorem}
The events $A=\{s\nrightarrow a\}$ and $B=\{t\nrightarrow b\}$ are positively correlated for $n\geq 4.$
\end{theorem}

\begin{proof}
From Lemmas \ref{sl} and \ref{lowerB} we get
\begin{eqnarray*}
&&\P(A\cap B)-\P(A)\P(B)=\P(A\cap B)-(\P(A))^2 \geq \\
&\geq & \left(\frac12 \right) ^{2n-4}\left( 3- \left(\frac12 \right) ^{2n-7}  - \left(\frac12 \right) ^{n-4}  -\left(1+3.2\cdot \left(\frac78 \right) ^{n-1}\right)^2 \right)\\
&>& 0 \text{ when } n\geq 13.
\end{eqnarray*}
The cases $4\leq n\leq 12$ were checked using Lemma \ref{lowerB} and the values of $\P(A)$ computed by recursion in \cite{AL}. The (rounded) values used are listed below.
\begin{center}
\begin{tabular}{l l}
$n$ & $\P(A)$\\
\hline
4&0.25\\
5&0.146484\\
6&0.076416\\
7&0.036942\\
8&0.017427\\
9&0.008309\\
10&0.004038\\
11&0.001988\\
12&0.000986
\end{tabular}
\end{center}

\end{proof}

We can also give an upper bound for $\P(A\cap B)$ to show that $\lim_{n\rightarrow \infty} \P(A\cap B)\cdot 2^{2n-4}=3$ and $\lim_{n \rightarrow \infty} \frac{\P(A\cap B)-\P(A)\cdot \P(B)}{\P(A\cap B)}=\frac23$ . These
statements are special cases of Theorems \ref{gnp2} and \ref{rel} below. 



\section{Random orientations of $G(n,p).$}
\label{gnp}

Let as usual $\gnp$ be the random graph in which every edge exists with probability $p$ independently of the other edges.  We also let every edge be directed in either way with equal probability independently of each other. We will call the corresponding random graph model $\ggnp$. For this section, let $x=p/2$ be the probability of one edge to exist and be directed in a certain way and let $y=1-x$ be the probability of an edge not to exist in a certain direction. We will adopt the usual notation $f\sim g$ to denote  that the quotient of $f$ and $g$ goes to a constant. 
In \cite{AJL} the following lemma was proven.

\begin{lemma}[Lemma 4.2 in \cite{AJL}]
\label{sl2ab} For any vertices $s,a$ in $\ggnp$ 
$$\P(A) \sim   2y^{n-1}.$$
\end{lemma}
 
Clearly, $\P(A)=\P(B)$. 
To find the relative correlation between $A$ and $B$ when $n$ approaches  infinity we need an estimate of $\P(A\cap B)$. 

A set $X$ of vertices in $K_n$ is said to be an inset (outset) if all existing edges from $[n]\backslash X$ are directed to (from) $X.$ Let $I^X$ be the event that $X$ is an inset. Let also 
$$Z_k=\bigcup_{\tiny{\begin{array}{c}X:s\in X\\ a\notin X\\ |X|=k\end{array}}}I^X \,\,\,\, \text{   and   } \,\,\,\,Z_k'=\bigcup_{\tiny{\begin{array}{c}X':t\in X'\\ b\notin X'\\ |X'|=k\end{array}}}I^{X'}.$$
Now we have $$\P(s\nrightarrow a)=\P(\bigcup_{k=1}^{n-1}Z_k)$$ and $$\P(A\cap B)=\P(s\nrightarrow a,t\nrightarrow b)=\P(\bigcup_{k=1}^{n-1}Z_k\cap\bigcup_{k=1}^{n-1}Z'_k).$$



\begin{theorem}
\label{gnp2}
For $p\in (0,1]$ we have
$$\lim_{n\rightarrow \infty}\frac{\P(A\cap B)}{ y^{2n-4}}= 4-p$$
\end{theorem}

\begin{rema+}
Exact computations indicate that this convergence is very slow for small $p$, see Figure \ref{F:h} in
Section \ref{S:rec}.
\end{rema+}

\begin{proof}
First note that 
$\P(A\cap B)=\P(\cup_{k=1}^{n-1}Z_k\cap\cup_{k=1}^{n-1}Z'_k)=  s_1+s_2+s_3-s_4$, where 
$s_1= \P(\cup_{k=3}^{n-3}Z_k\cap\cup_{k=1}^{n-1}Z'_k)$, $s_2=\P(\cup_{k=1}^{n-1}Z_k\cap\cup_{k=3}^{n-3}Z'_k)$,
$s_3= \P\left(\left(\cup_{k=1}^{2}Z_k \cup_{k=n-2}^{n-1}Z_k\right)\cap \left(
\cup_{k=1}^{2}Z'_k\cup_{k=n-2}^{n-1}Z'_k\right)\right)$ and  
$s_4= \P(\cup_{k=3}^{n-3}Z_k\cap\cup_{k=3}^{n-3}Z'_k)$. 
By symmetry $s_1=s_2$ and clearly $s_4<s_1$. We will write $\P_N(I^X)$ for $\P(I^X)$ with $|X|=N.$ 
We show that $s_1,s_2,s_4$ are negligible compared to $s_3$, and give an estimate  of $s_3$. 
Starting with $s_1,$ first note that $\P_k(I^X)=y^{k(n-k)}$ and if $k<l\leq \frac{n}{2}$ we have $\P_l(I^{Y})<\P_k(I^X).$  
This gives us
\begin{eqnarray*}
s_1 &=& \P(\bigcup_{k=3}^{n-3}Z_k\cap\bigcup_{k=1}^{n-1}Z'_k)\\
&\leq& \P(\bigcup_{k=3}^{n-3}Z_k)\\
&\leq& \sum_{k=3}^{n-3}{n \choose k-1}y^{k(n-k)}\\
&\leq& 2\cdot \sum_{k=3}^{K-1}{n \choose k-1}y^{k(n-k)} + \sum_{k=K}^{n-K}{n \choose k-1}y^{k(n-k)}
\end{eqnarray*}

Now, since $p$ is fixed we may fix $K$ such that $y^K < \frac{y^3}{2}$. 
The sum $ \sum_{k=3}^{K-1}{n \choose k-1}y^{k(n-k)}$ is finite and it is $O(y^{3(n-3)})$ 
which is very small compared to $y^{2n}$, and hence negligable. Further we get
\begin{eqnarray*}
 \sum_{k=K}^{n-K}{n \choose k-1}y^{k(n-k)} &< & 2^n \cdot y^{K(n-K)}\\
&< & 2^n \left( \frac{y^3}{2} \right)^{n-K}=O(y^{3n}).
\end{eqnarray*} 

That is $s_1  \sim  o(y^{2n})$ and analogously so is $s_2$ and $s_4$.

To estimate  $s_3,$ first consider $\P(Z_1\cap Z'_2)$ as an example. In this case no edges will be directed from $s$. For the inset
$X'$ we have two subcases, either it  contains $s$ and $t$ or $t$ and another vertex (different from $s,b).$ In the first case we get a total of $y^{2n-3}$, and for the second case we can choose $X'$ in $n-3$ ways and no edges will be directed from $X'$, this gives us $(n-3)y^{3n-9}(1-p)^2.$ 
In the computations below it will always be the case that if three or more vertices are involved.  Then
the probability will be negligable, i.e. $o(y^{2n})$. 

We get four contributing cases which can be reduced to two by symmetry. 
\begin{itemize}
 \item[(1)] $\P((Z_1\cup Z_2)\cap (Z'_1\cup Z'_2))=y^{2n-4}+o(y^{2n})$. 
\item[(2)] $\P((Z_{n-1}\cup Z_{n-2})\cap (Z'_{n-1}\cup Z'_{n-2}))=y^{2n-4}+o(y^{2n})$. 
\item[(3)] $\P(Z_{1}\cap Z'_{n-1})=y^{2n-3}$. 
\item[(4)] $\P(Z_{n-1}\cap Z'_{1})=y^{2n-3}$. 
\end{itemize}

For $(1)$ we see that if any other vertex than $s$ and $t$ is in the insets  for  $s$ and $t$ we will have conditions on at least $3n-9$ edges and thus a probability of size $o(y^{2n})$. All the interesting cases are when we have no restriction on the possible edge between $s$ and $t$, and no edge must be directed from $s,t$ to any other vertex. Note that our example above is a subset of this case.
Case $(2)$ is symmetric to $(1)$.

For $(3)$ no edge may be directed from $s$ or to $b$, which imposes conditions on $2n-3$ edges. Case $(4)$ is symmetric to $(3)$.
One can easily check that the remaining six possibilities, four cases symmetric to $Z_1\cap Z'_{n-2}$ and two cases symmetric to $Z_2\cap Z'_2$, all have probabilities of size $o(y^{2n})$ and can hence be ignored. 
%

All together we end up with  $2y^{2n-4}+ 2y^{2n-3}+  o(y^{2n})= 2y^{2n-4}(1+(1-\frac{p}{2}))+o(y^{2n})= y^{2n-4}(4-p)+o(y^{2n})$. 
\end{proof}

\begin{theorem}
\label{rel}
For fixed $p\in [0,1]$
$$\lim_{n\rightarrow \infty}\frac{\P(A\cap B)-\P(A)\P(B)}{\P(A\cap B)}=\frac{p(3-p)}{4-p}$$
\end{theorem}
\begin{proof}
Follows from Lemma \ref{sl2ab}  and Theorem \ref{gnp2}.
\end{proof}

\begin{corollary}
For a fixed $p\in (0,1]$, the correlation between $A$ and $B$ is always positive for sufficiently large $n$.
\end{corollary}

We believe that something stronger is true and we offer the following conjecture, which is supported by our calculations in 
Section \ref{S:rec}.

\begin{conjecture}\label{C:positive}
For any $n\ge 4$ and $p\in (0,1]$, the events $\{s\to a\}$ and $\{t\to b\}$ are always positively correlated.
\end{conjecture}

\section{Random orientations of $G(n,m)$}
\label{gnm}
In this section we study the same problem on the random graph $\gnm$ where each simple graph with $m$ edges and $n$ vertices is equally likely. We will also here let every edge have an independent direction and call the combined probability space $\ggnm$. 
 Again, let $y=1-\frac{p}{2},$ let further $q(l)=q(l;n,m)$ be the probability that $l$ fixed edges in $K_n$ does not exist in $\ggnm$ with given directions. 
In $\ggnp$ this corresponds to $y^l.$ If nothing else is written the graph considered in this section is always 
$\ggnm$. 

In \cite{AJL} the following lemma was prooven.

\begin{lemma} [Janson, Lemma 3.2 in \cite{AJL}] 
\label{gnme}
 Suppose that $0\leq m=m(n) \leq {n \choose 2}.$ Then with $p=p(n)=m(n)/ {n \choose 2},$ as $n\rightarrow \infty$,
$$q(l;n,m) \sim y^l \exp \left( -\left( \frac{l}{n} \right) ^2 \frac{p(1-p)}{(2-p)^2}\right),$$
and for any $l, n, m$ we have $q(l;n,m) \leq q'(l;n,p).$
\end{lemma}

This lemma together with the proof of Theorem \ref{gnp2} gives us an analogue result 
of Theorem \ref{gnp2} for $\ggnm$.

\begin{theorem}
\label{p(ab)Gnm}
In the case of $G(n,m)$ for fixed $0<p<1$ we have 
\begin{eqnarray*}
\P(A\cap B) &\sim& 2y^{2n-4}\exp \left( - 4\frac{p(1-p)}{(2-p)^2}  \right) \\
&+&2y^{2n-3}\exp \left( - 4\frac{p(1-p)}{(2-p)^2}  \right) 
\end{eqnarray*}
\end{theorem}
 
Also we need the following lemma.

\begin{lemma}[Lemma 4.3 in \cite{AJL}]
\label{p(a)Gnm}
For fixed $0<p<1$ 

$$\P(A) \sim 2y^{n-1} \exp \left( - \frac{p(1-p)}{(2-p)^2}  \right).$$
\end{lemma}

We are now ready to state and prove the main theorem of this section.

\begin{theorem}
For  fixed $0<p<1$  and sufficiently large $n$, the events $A$ and $B$ are positively correlated and the relative covariance is
$$\sim 1-\frac{2\left(1-\frac{p}{2}\right )^2}{2-\frac{p}{2}} \cdot \exp \left( 2\frac{p(1-p)}{(2-p)^2}  \right).$$
\end{theorem}

\begin{proof}
We rewrite the relative covariance as 
$$\frac{\P(A\cap B)-\P(A)\P(B)}{\P(A\cap B)}=1-\frac{\P(A)\P(B)}{\P(A\cap B)}.$$
As $n$ approaches  $\infty,$ Theorem \ref{p(ab)Gnm} and Lemma \ref{p(a)Gnm} gives
\begin{eqnarray*}
\frac{\P(A)\P(B)}{\P(A\cap B)}&\sim&  \frac{4y^{2n-2}  \exp \left(-2 \frac{p(1-p)}{(2-p)^2}  \right)  }{  2y^{2n-4}\exp \left(-4 \frac{p(1-p)}{(2-p)^2}  \right) \left( 2-\frac{p}{2}\right)}  \\
&=& \frac{2 \left(1-\frac{p}{2}\right )^2 }{\left( 2-\frac{p}{2}\right)} \exp \left( 2\frac{p(1-p)}{(2-p)^2}  \right)
\end{eqnarray*}
Let us denote this expression by $f$. It remains to prove that $f$ is less than one when $0<p<1$. This can be proven by using the derivative of $f.$ We have that 
$$f'(p)= e^{\frac{2 (1-p) p}{(2-p)^2}} \frac{p^3+4p^2-8}{(4-p)^2 (2-p)}.$$ 
The theorem follows since the derivative is negative in this interval and $f(0)=1.$


\end{proof}
We conjecture the covariance to be positive at all times.
\begin{conjecture}
The events $A$ and $B$ are positivelly correlated in $\ggnm$ for all $p$ and all $n.$
\end{conjecture}
Note that the covariance of $\ggnp$ is always less than the covariance of $\ggnm$ (see \cite{AJL}). So the conjecture would also imply the correlation to be positive in $\ggnp.$


\section{Exact recursion in $\ggnp$.}
\label{S:rec}
In this section we will give an exact recursion to compute 
\[ \P_{\ggnp}(a\notto s, t\notto b).\] 
Together with the recursion given for $\fpn:=\P_{\ggnp}(a\notto s)$ in 
\cite{AJL} we will be able compute the covariance for $n$ as a rational function in $p$. 
Our computations for $n\le 34$,
using Maple, supports our Conjecture \ref{C:positive} that the covariance is always positive, see 
Figure \ref{F:Corr}.

\begin{figure}[htbp]
\begin{center}
\epsfig{file=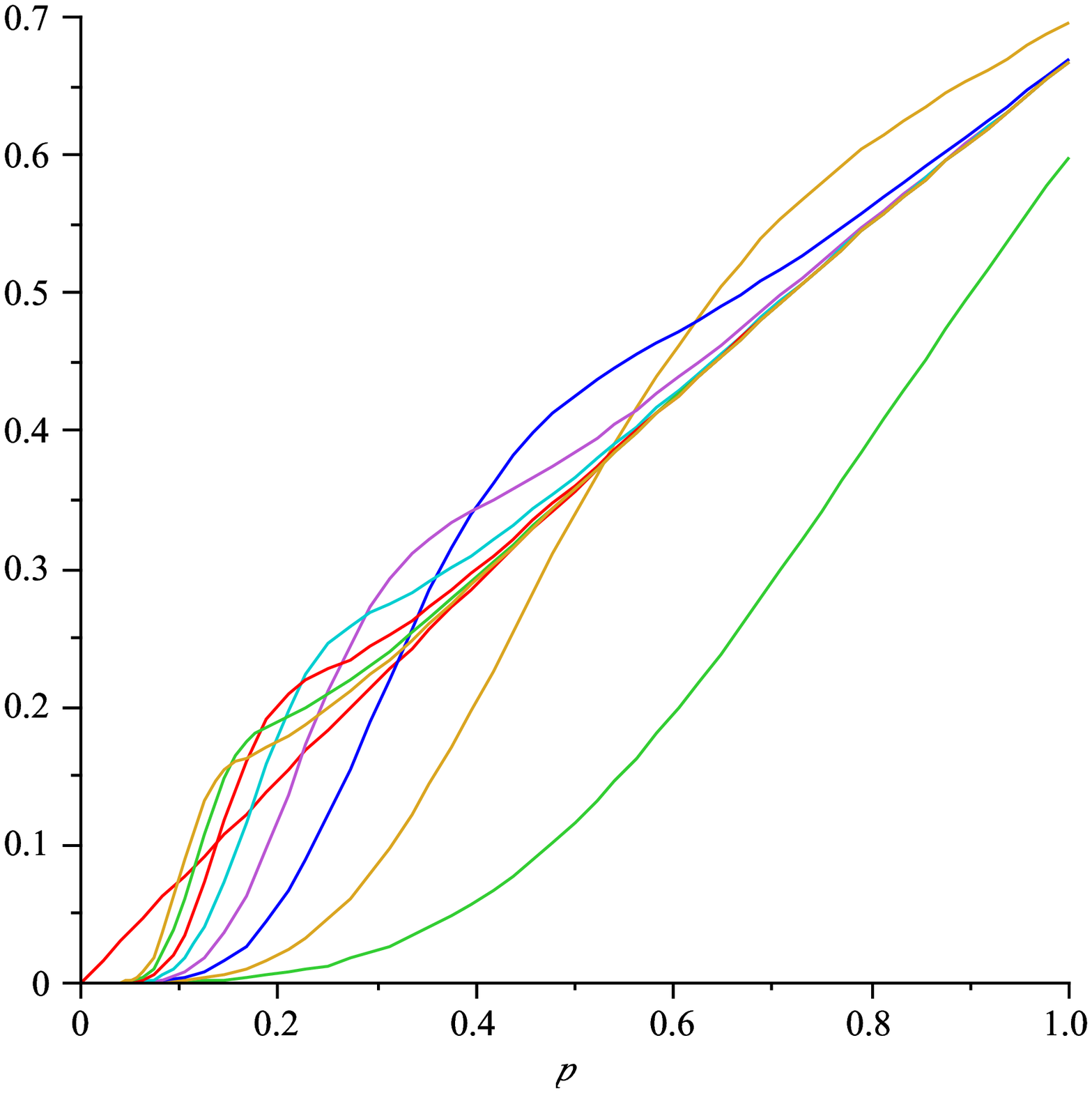, height=8cm}
\end{center}
\caption{The relative covariance $\frac{\P(a \notto s, t \notto b)-\P(a \notto s)\P(t \notto b)}{\P(a \notto s, t \notto b)}$ in $\ggnp$
for going from right to left $n=6\text{ (green)} ,8,10,12,14,16,18,20,22\text{ (blue)} $, 
and the asymptote $p(3-p)/(4-p)$. 
All curves are positive for $0<p\le 1$.} 
\label{F:Corr}
\end{figure}

\begin{figure}[htbp]
\begin{center}
\epsfig{file=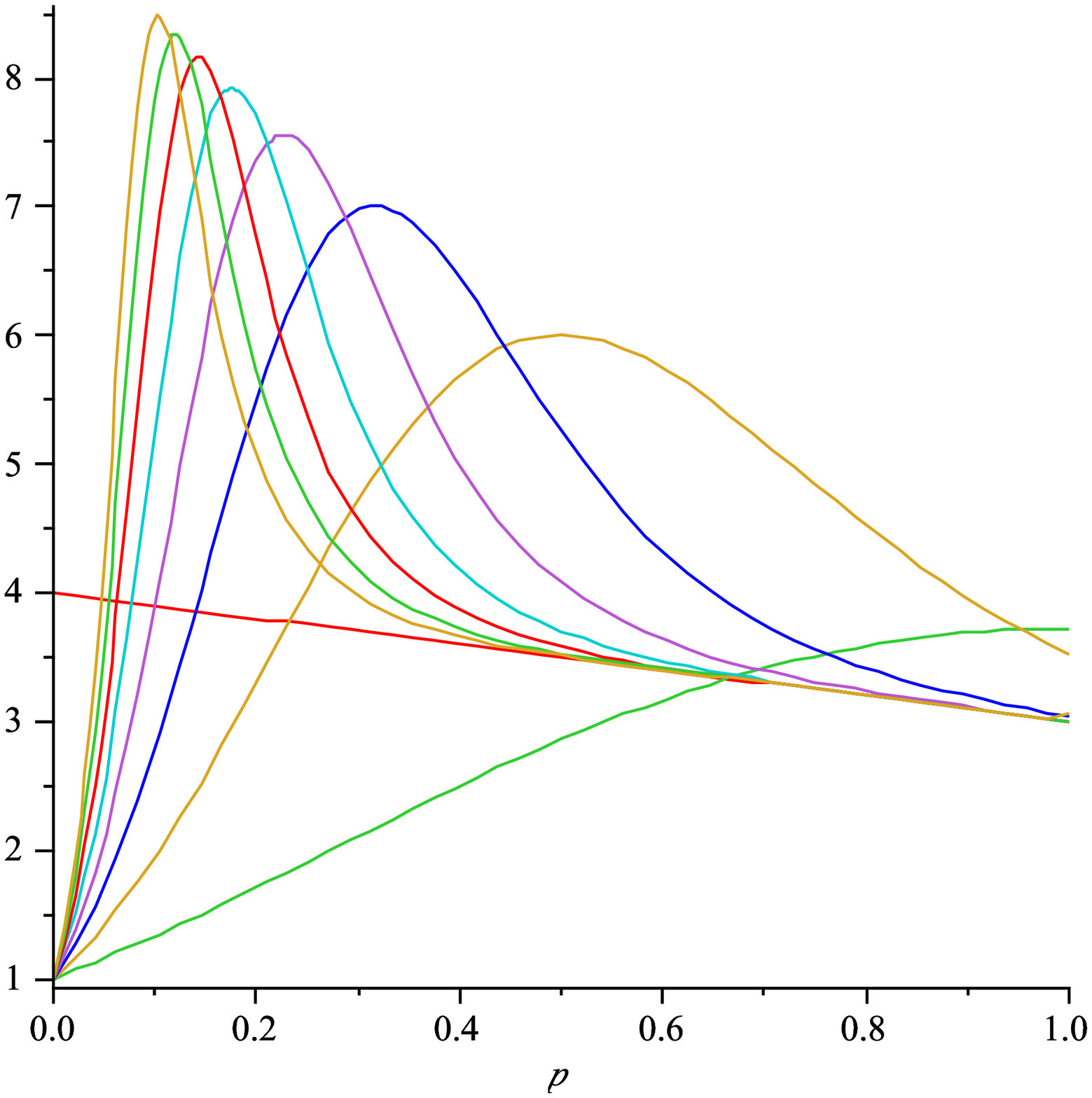, height=6cm}
\end{center}
\caption{Plots of $\frac{\P(a \notto s, t \notto b)}{y^{2n-4}}$ in $\ggnp$
for going from right to left $n=6\text{ (green)} ,8,10,12,14,16,18,20,22\text{ (blue)} $, 
and the asymptote $4-p$. 
See Theorem \ref{gnp2}.} 
\label{F:h}
\end{figure}

For a vertex $v\in V(G)$, let $\hpil{C}_v\subseteq V(G)$ be the (random) set
of all vertices $u$ for which there is a directed path from $v$ to $u$. We
say that $\hpil{C}_v$ 
is the {\bf out-cluster} from $v$. Let analogously 
the {\bf in-cluster}, $\vpil{C}_v\subseteq V(G)$ be the (random) set of all vertices
$u$ for which there is a directed path from $u$ to $v$. Note that we will
use the convention that $v\in \vpil{C}_v\cap\hpil{C}_v$. 
Let as before $y:=1-p/2$ 
be the probability that an edge does not exist with
a certain direction, and let $q:=1-p$ be
the probability that there is no edge
at all.

For $n\ge 1$, 
$s\in {\S}\subseteq [n]$ and $|{\S}|=k$ define:
\[ d_p(n,k):=\P_{\ggnp}(\hpil{C}_s={\S}),\]
where in particular $d_p(1,1)=1$.
A recursion to compute $d_p(n,k)$ as a polynomial in $p$ was given in \cite{AJL}.

\begin{lemma}[Lemma 5.1 in \cite{AJL}]
 We have the following recursions
$$d_p(n,k)=d_p(k,k)y^{k(n-k)}, \text{ for } n>k\geq 1,$$
and
$$d_p(k,k)= 1- \sum_{i=1}^{k-1} {k-1 \choose i-1}d_p(i,i)y^{i(k-i)}.$$
\end{lemma}

 Note that, by symmetry, also $\P_{\ggnp}(\vpil{C}_s={\S})=d_p(n,k)$.
\medskip




It turns out that the following quantity is possible to compute recursively and enables us to compute $\hpn$. 
For $n\ge 2$,
$ t\in {\T}\subseteq [n]$, $ a\in {\A}\subseteq [n]$ with $|{\T}|=\tau,
|{\A}|=\al$ and  
$|[n]\setminus ({\A}\cup {\T})|=r$ define:
\[ N_p(n,\tau,\al,r):=\P_{\ggnp}(\hpil{C}_t={\T}, \vpil{C}_a={\A}),\]
where in particular $N_p(2,2,2,0)=x$ and $N_p(2,1,1,0)=y$. 

We will use the variable $j$ for the size of the intersection $|\A\cap\T|$. If there is any intersection between $\A$ and $\T$ then $a,t\in \A\cap \T$, so in particular $j=\al+\tau-(n-r)$ can never be 1.

\begin{theorem} \label{T:h} 
We have the following recursions for $N_p$,
where $\tau+\al>n-r\ge \tau,\al$ and $\tau,\al \ge 1$
\begin{romenumerate}
\item\label{lg1} 
$\displaystyle 
N_p(n,\tau,\al,r)=N_p(n-r,\tau,\al,0)q^{r(r+\tau+\al-n)}y^{r(2n-2r-\tau-\al)}, \quad \text{for $r>0$}$,
\item\label{lg2} 
$\displaystyle N_p(n,\tau,\al,r)=N_p(n,\al,\tau,r)$,
\item\label{lg3} 
$\displaystyle 
N_p(n,\tau,\al,0)=\sum_{\zeta=1}^{n-\tau}\binom{n-\tau-1}{\zeta-1}N_p(n-\zeta,\tau,\al-\zeta,0)d_p(\zeta,\zeta)q^{(\zeta-1)(\al+\tau-n)}\cdot$\\
$\displaystyle 
y^{(\zeta-1)(2n-\tau-\al-\zeta)}\big(y^{\tau}-y^{2n-\al-\tau-\zeta}q^{\al+\tau-n}\big) ,\quad \text{for $n>\tau, n\ge \al\ge 2, j\ge 2$}$,
\item\label{lg4} 
$\displaystyle 
N_p(n,\tau,\al,0)=\sum_{\zeta=1}^{\al-1}\binom{\al-2}{\zeta-1}N_p(n-\zeta,\tau,\al-\zeta,0)d_p(\zeta,\zeta)\cdot$\\
$\displaystyle 
y^{(\zeta-1)(\tau+\al-\zeta)}y^{\tau}\left(1-y^{\al-\zeta}\right) ,\quad \text{for 
$\al\ge 2, j=0$ i.e. $n=\tau+\al$}$,
\item\label{lg5} 
  $\displaystyle 
N_p(n,n,n,0)=$\\
 $\displaystyle {}
1-\sum_{j=2}^{n-1}\binom{n-2}{j-2}\sum_{\tau=j}^{n}\binom{n-j}{\tau-j}\sum_{\al=j}^{n-\tau+j}\binom{n-\tau}{\al-j}N_p(n,\tau,\al,n-\al-\tau+j)$\\
$\displaystyle 
-\sum_{\tau=1}^{n}\binom{n-2}{\tau-1}\sum_{\al=1}^{n-\tau}\binom{n-\tau-1}{\al-1}N_p(n,\tau,\al,n-\al-\tau),$

\end{romenumerate}
\end{theorem} 

\begin{proof} For the first equation we have $r>0$, thus $[n]\setminus ({\A}\cup {\T})$ is non-empty and
no vertex in that set must not have any edge directed
 to ${\A}$ or from ${\T}$. Hence there must be no edge at all to
 ${\A}\cap {\T}$, which gives probability $q^{|[n]\setminus {\A}\cup
 {\T}|\cdot |{\A}\cap {\T}|}=q^{r(r+\tau+\al-n)}$. There must not be any edge 
directed to $({\A}\setminus {\T})$  and there must not be any edge directed from $({\T}\setminus {\A})$. This gives the probability of $y^{|[n]\setminus {\A}\cup {\T}|\cdot |({\A}\setminus {\T})\cup ({\T}\setminus {\A})|}=
 y^{r(2n-2r-\tau-\al)}$.

The second equation is obtained from the symmetry of reversing all directions and switching the roles of 
$a$ and $t$.
 
 For equation \ref{lg3} and \ref{lg4}, 
 we pick a vertex $z\in {\A}\setminus {\T}$, such a vertex exists by the assumption $n>\tau$ and $r=0$.
Let $G$ be any directed
graph on $n$ vertices with $\hpil{C}_t={\T}$ and $\vpil{C}_a={\A}$. If we
remove vertex $z$ and all its edges from $G$ the resulting graph will still
have $\hpil{C}_t={\T}$ since $z\notin {\T}$, whereas $\vpil{C}_s=\A \backslash {\Z}$, for
some $\Z\subseteq \A\backslash\{a\}$ such that ${\Z}\cap {\T}=\emptyset$. This follows from the fact that
the vertices in $\Z$ are those that have a path to $a$ only via $z$ and no vertex in $\T$ has a directed path leading to $z$ by assumption.
Let $\zeta=|{\Z}|$ and sum over all possible ${\Z}$.  
 The probability is $N_p(n-\zeta,\tau,\al-\zeta,0)$ that the subgraph on $[n]\setminus
 \Z$ is as needed. The subgraph on $\Z$
 must have $\vpil{C}_z={\Z}$ which has probability
 $d_p(j,j)$. Let us first consider equation \ref{lg3} when $j=\tau+\al-n\ge 2$.

 There must not be any edge between ${\T}\cap {\A}$ and $\Z\setminus \{z\}$, since the vertices of the latter do not belong to ${\T}$ and have all directed paths via $z$. 
 This gives a factor $q^{(\zeta-1)(\al+\tau-n)}$. 
No vertex of  ${\Z\setminus \{z\}}$ can have an edge to ${\A}\setminus ({\T}\cup \Z)$ or
 from ${\T}\setminus {\A}$ , which gives a factor $y^{(\zeta-1)(2n-\tau-\al-\zeta)}$. 
 Finally, we must consider the edges of $z$. The main condition is that there must not be any edge 
 from $\T$ to $z$. However, there must be at least one edge edge directed from $z$ to $\A\setminus \Z$. This 
 give the last factor.
 The case of equation \ref{lg4} when $j=0$ is easier and obtained similarly.

Equation \ref{lg5} follows from the fact that for fixed $n$
\[\sum_{\T,\A :a\in {\A},t\in {\T}\subseteq [n]} \P_{\ggnp}(\hpil{C}_t={\T},\vpil{C}_a={\A})=1.\] 
Here $j=|{\A}\cap {\T}|$ and recall that $j=1$ is not an option.

\end{proof}

\begin{theorem}
We have the following expression for $ \P_{\ggnp}(a\notto s, t\notto b).$ 
\begin{multline*}
   \P_{\ggnp}(a\notto s, t\notto b) =\sum_{j=2}^{n-2}\binom{n-4}{j-2}\cdot \\
\qquad  \Biggl( \;
\sum_{\tau=j}^{n-2}\binom{n-2-j}{\tau-j}\sum_{\al=j}^{n-\tau+j-1}\binom{n-\tau-1}{\al-j}N_p(n,\tau,\al,n-\al-\tau+j)\\
+\sum_{\tau=j+1}^{n-1}\binom{n-2-j}{\tau-j-1}\sum_{\al=j}^{n-\tau+j}\binom{n-\tau}{\al-j}N_p(n,\tau,\al,n-\al-\tau+j)
\Biggr)\\
\qquad
+\sum_{\tau=1}^{n-3}\binom{n-4}{\tau-1}\sum_{\al=1}^{n-\tau-1}\binom{n-\tau-2}{\al-1}N_p(n,\tau,\al,n-\al-\tau)\\
+\sum_{\tau=2}^{n-2}\binom{n-4}{\tau-2}\sum_{\al=j}^{n-\tau}\binom{n-\tau-1}{\al-1}N_p(n,\tau,\al,n-\al-\tau),
\end{multline*}
\end{theorem}

\begin{proof}
The equation for $\P_{\ggnp}(a\notto s, t\notto b)$  is obtained by summing over all possible pairs ${\A},{\T}$
such that $s\notin {\A}, b\notin {\T}$. Again $j=|{\A}\cap {\T}|$ and the
formula is split into four cases depending on if $s\in {\T}$ or not and if $j=0$ or not. 
\end{proof}

Note that in $\ggnp$ the functions $\P(s \notto a)$ and $\P(s \notto a, t
\notto b)$  
are polynomials in $p$ and hence continuous.


\section{The Quenched model}
\label{quenched}
For the quenched version the correlation between $A$ and $B$ is computed for each graph in $G(n,p) \, (G(n,m))$ in the probability space of edge orientations and then the expected value is taken over all graphs. 

We computed the covariance between $A$ and $B$ for $G(n,p)$ as a function over $p,$ in both the annealed and the quenched model for $n\leq 6.$ The two cases looks quite similar, see Figure \ref{bild6}.  Note that for $n\leq 6$ the covariances are positive
also for small $p$ and we conjecture it to be positive for all $n.$ This differs from the behavior for the similar problem studied in
Section 9 in \cite{AJL}. It would also be intresting to find an analouge to Theorem \ref{rel} for the quenched model.

\begin{figure}[h] 
\centering
\includegraphics[width=12cm]{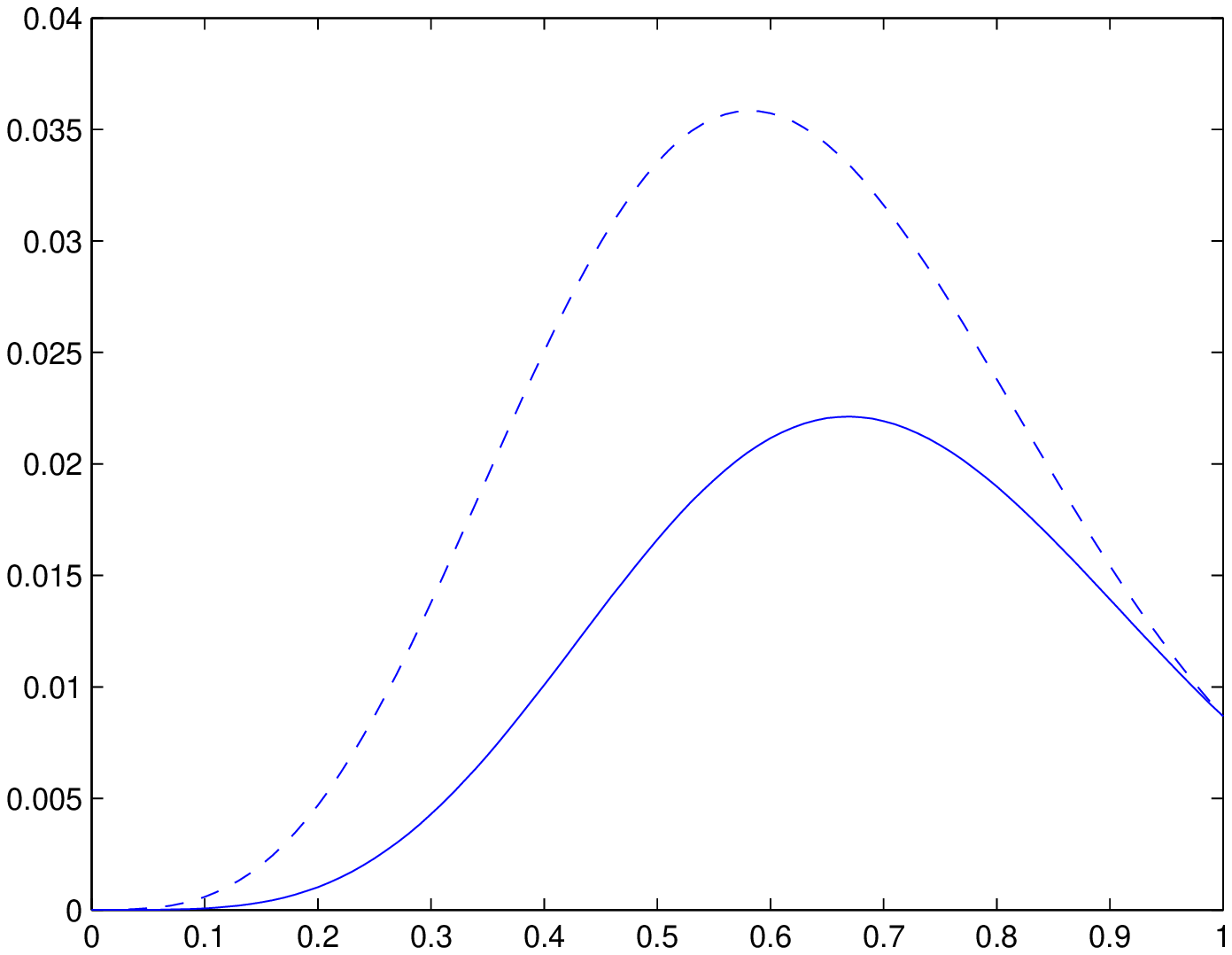} 
\caption{The covariance for $G(6,p)$. The dashed curve represents the annealed case and the continous one the quenched case.}\label{bild6}
\end{figure}

\newcommand\webcite[1]{\hfil  
   \penalty0 
\texttt{\def~{{\tiny$\sim$}}#1}\hfill\hfill}

\newcommand\arXiv[1]{\webcite{arXiv:#1.}}


\begin{thebibliography}{ABX}

\bibitem{AL} Sven Erick Alm and Svante Linusson, A counter-intuitive correlation in a random tournament, 
{\em Combin. Probab. Comput.} {\bf 20} (2011), no 1, 1--9.

\bibitem{AJL} Sven Erick Alm, Svante Janson and Svante Linusson, Correlations for paths in random orientations of $G(n,p)$ and $G(n,m)$, {\em Random Structures and Algorithms } {\bf 39}, No. 4, 486--506 (2011). 

\bibitem{G00} Geoffrey R.\ Grimmett, Infinite paths in randomly oriented lattices,  {\em Random Structures and Algorithms} {\bf 18}, 
No 3, (2001), 257 -- 266. 

\bibitem{M81} Colin McDiarmid, General percolation and random graphs, {\em
  Adv. in Appl. Probab.} {\bf 13} (1981), 40--60.
\bibitem{L1} Svante Linusson, On percolation and the bunkbed conjecture, {\em Combinatorics, Probability and Computing}
{\bf 20} no 01, pp. 103-117 (2011). 



\end{thebibliography}
\end{document}